\documentclass{amsart}

\usepackage{graphicx}
\usepackage{epstopdf}
\usepackage{amsmath}
\usepackage{amsthm}
\usepackage{amssymb}

\newtheorem{theorem}{Theorem}[section]
\newtheorem{lemma}[theorem]{Lemma}

\newtheorem{cor}[theorem]{Corollary}

\theoremstyle{definition}
\newtheorem{dfn}[theorem]{Definition}

\title{The Stable Rank of Diagonally Constructed ASH Algebras}
\author{James Lutley}
\date{\today}
\address{Department of Mathematics\\
  University of Toronto\\
  Toronto, Ontario, Canada M5S 2E4}
\email{james.lutley@mail.utoronto.ca}

\begin{document}

\begin{abstract}
We introduce a class of recursive subhomogeneous algebras that we call diagonal subhomogeneous and we give a notion of diagonal maps between these algebras. We show that any simple limit of diagonal subhomogeneous algebras with diagonal maps has stable rank one. As an application we show that for any minimal homeomorphism of a compact Hausdorff space the associated crossed product has stable rank one.\end{abstract}
\maketitle

Given maps $\sigma_1,\ldots,\sigma_k: Y\to X$, where $X,Y$ are compact Hausdorff spaces, we may readily construct a homomorphism from $C(X, M_n)$ to $C(Y, M_{nk})$ by sending $f\in (X, M_n)$ to $\mathrm{diag}(f\circ \sigma_1, \ldots, f\circ \sigma_k)$, and we say this map is diagonal. Diagonal maps and their generalizations have played a prominent role in constructing examples of AH algebras, including the Goodearl construction \cite{good} and Villadsen's examples \cite{vil1}\cite{vil2}. It was shown by Elliott, Ho, and Toms \cite{ho} that any simple AH algebra constructed from diagonal maps has stable rank one. Our present work uses the techniques developed in that paper and its antecedents and applies them to a certain class of subhomogeneous algebras. Maps between these algebras cannot be said to be diagonal in the above sense: we do not have perfect analogues of the single valued eigenvalue maps $\sigma_i$. However, by restricting our attention to a class of subhomogeneous algebras in which there is a rigid notion of a diagonal, we can readily define a class of maps which send each point in the spectrum of the range algebra to an ordered list of eigenvalues in the domain algebra. Thus we give a notion of diagonal maps, and it turns out these are enough to once again obtain stable rank one. This class of algebras is inspired by the orbit breaking algebras introduced by Q. Lin following from work of Putnam \cite{putnam}. Using results of Archey and Phillips \cite{archey}, we are able to show that dynamical cross products which come from a single minimal homeomorphism of a compact Hausdorff space have stable rank one.

This paper is constructed as follows: in section 1,  we introduce the class of diagonal subhomogeneous algebras that we are able to work with. We define diagonal maps between these algebras and show that orbit-breaking algebras are DSH algebras with diagonal inclusions. In section 2, we recall the families of unitary paths between permutation matrices which were used in \cite{ho}. In section 3 we show how to construct unitary elements within DSH algebras using the matrices of section 2. We then detail in section 4 how the assumption of simplicity on the limit algebra allows us to ensure that any non-invertible element is very near to a function in a sequence algebra which has a non-invertible image in any representation. We can then multiply by the unitaries we have constructed to obtain a nilpotent element. We employ a trick of R\o rdam that has become a standard argument for proving stable rank one: since nilpotent elements lie in the closure of the invertibles, if $ufv$ is nilpotent and $u,v$ are unitaries then $f$ is approximately invertible. In section 5 we prove our main theorem, that simple diagonal limits of diagonal subhomogenous algebras have stable rank one. It then follows from \cite{archey} that dynamical cross products coming from minimal homeomorphisms on compact spaces have stable rank one.

\section{Diagonal Subhomogeneous Algebras}

We define a well-behaved class of unital subhomogeneous algebras that behaves like a diagonal version of recursive subhomogeneous (RSH) algebras as defined in \cite{phillips2}.

\begin{dfn} The class of {\em recursive subhomogeneous algebras} is the smallest one that satisfies the following conditions.
\begin{enumerate} \item If $X$ is a compact Hausdorff space and $n\in \mathbb{N}$ then $C(X,M_n)$ is RSH.
\item If $A$ is RSH and $X$ is a compact Hausdorff space with a closed subset $Y$ and $\phi: A \rightarrow C(Y,M_n)$ is any unital homomorphism and $\rho: C(X,M_n)\rightarrow C(Y,M_n)$ is the restriction homomorphism then
$$A\oplus_{C(Y,M_n)} C(X,M_n)=\{ (a,f)\in A\oplus C(X,M_n): \phi(a)=\rho(f)\}$$
is RSH.\end{enumerate}\end{dfn}

It follows that we can write any RSH algebra $A$ in terms of a finite composition sequence using spaces $X_1,\ldots,X_l$ that contain closed subspaces $Y_1,\ldots,Y_l$; integers $n_1,\ldots,n_l$; and unital homomorphisms $\phi_1,\ldots,\phi_{l-1}$ using the above notation
$$A=\Big{[} \cdots [[ C_1 \oplus_{C_2^\prime} C_2]\oplus_{C_3^\prime} C_3]\cdots\Big{]} \oplus_{C_l^\prime} C_l$$
where $C_i= C(X_i,M_{n_i})$ and $C_i^\prime= C(Y_i,M_{n_i})$. We refer to $l$ as the length of the composition sequence.

Given a subalgebra $A\subseteq \bigoplus_{i=1}^l C(X_i,M_{n_i})$ and any $f\in A$, we will denote by $f_i$ the summand of $f$ in $C(x_i,M_{n_i})$. 

\begin{dfn} Suppose we have compact Hausdorff spaces $X_1,\ldots,X_l$, closed subspaces $Y_i\subseteq X_i$, dimensions $n_1,\ldots,n_l$, $C^*$-algebras $A_i\subseteq \bigoplus_{j=1}^i C(X_j,M_{n_j})$ such that $A_1=C(X_1,M_{n_1})$ and homomorphisms $\phi_i:A_i\to C(Y_{i+1},M_{n_{i+1}})$ for $i\in\{1,\ldots,l-1\}$ such that for each $y\in Y_{i+1}$ there are corresponding points $x_1,\ldots,x_t$ where $x_j\in X_{i_j}$ such that $\phi_i(f)_{i+1}(y)=\mathrm{diag}(f_{i_1}(x_1),\ldots,f_{i_t}(x_t))$ and $A_{i+1}=A_i\oplus_{C(Y_{i+1},M_{n_{i+1}})} C(X_{i+1},M_{n_{i+1}})$, then we say $A_l$ is {\em diagonal subhomogeneous} (DSH).\end{dfn}

Note that we can always assume that when $x\in Y_i$, each of the corresponding points $x_1,\ldots,x_t$ lies in $X_{i^\prime}\setminus Y_{i^\prime}$ for some $i^{\prime}<i$. We can do this by taking the list given by diagonality $x_{1}^\prime, \ldots x_{t}^\prime$, and whenever $x_{i}\in  Y_{i^\prime}$, we can replace it with the sequence of points for $x_i$ in the construction of $\phi_{i^\prime-1}$. Iterating, we eventually obtain a list as desired. 
Furthermore, we can check in the definition that if $Y_i$ has a non-empty interior, we may delete $\mathrm{int}(Y_i)$ without changing the algebra. That is, if $X_{i+1}^\prime= X_{i+1}\setminus \mathrm{int}(Y_i)$ and $Y_i^\prime=Y\setminus\mathrm{int}(Y_i)$, then 
$$\{ (a,f) \in A \oplus C(X_{i+1}^\prime, M_{n_{i+1}}) : \forall j, y\in Y_i^\prime, f_{i+1}(y) = \mathrm{diag}( a_{i_1}(x_{y_1}),\ldots,a_{i_s} (x_{y_s}) )\}$$
is isomorphic to the algebra constructed in the definition. Therefore we may assume each set $Y_i$ has an empty interior.

We will show that the orbit breaking algebras described by Q. Lin are examples of DSH algebras. 

In the DSH construction, for each space $i\in \{1,\ldots,l\}$ and each $k\in 1,\ldots, n_i$, we may define $B_{i,k}$ to be the subset of $X_i$ such that a new representation in the diagonal decomposition begins at line $k$. Every point in $X_i$ lies in $B_{i,1}$. For $k>1$, $B_{i,k}\subseteq Y_i$. When $y\in Y_i$ with corresponding points $x_{1}, \ldots ,x_{s}$ in  $X_{i_1} ,\ldots, X_{i_s}$, $y\in B_{i,k}$ for $k=n_{i_1}+1, n_{i_1}+n_{i_2}+1$ and so on.  Note that $n_1$, the smallest dimension of any representation of $A$, also gives a restriction on the sets $B_{i,k}$ in which a given point can appear. If $y\in B_{i,k}$, then $y\notin B_{i,k+1},\ldots,B_{i,k+n_1-1}$.

If $C$ is an $n\times n$ matrix, let us say $C$ has a block point at position $k$ if $C_{i,j}=0$ whenever either $i\geq k$ and $j<k$ or $i<k$ and $j\geq k$.

\begin{lemma}Suppose that $A$ is a DSH algebra. For each $i,k$, $x\in X_i$ lies in $B_{i,k}$ if and only if for every $f\in A$, $f_i(x)$ has a block point at position $k$. Every such set $B_{i,k}$ is closed.\label{blockchar} \end{lemma}
\begin{proof} It is clear that if $x\in B_{i,k}$, then every $f_i(x)$ has a block point at position $k$ for every $f$.
Suppose $x\in X_i$ does not lie in $B_{i,k}$. If $x\notin Y_i$, then because $Y_i$ is closed, we may easily find a function $f\in A$ for which $f_i(x)$ does not have a block point at $k$. Simply let $f_{i^\prime}$ be $0$ on $X_{i^\prime}$ for all $i^\prime < i$ and choose an appropriate function supported on $X_i \setminus Y_i$. Otherwise $x\in Y_i$ where we assume the that $k$ lies in the middle of some block beginning at $k^\prime <k$, so there exists a point $x^\prime \in X_{i^\prime}$ such that for every $f\in A$, $f_i(x)$ has a block point at position $k-k^\prime$, where $x^\prime \notin B_{i^\prime, k-k^\prime}$. We may therefore always reduce this to the case that $x\notin Y_i$.
Since the set of points $x\in X_i$ at which any $f_i(x)$ has blockpoints at position $k$ is evidently closed, it follows that $B_{i,k}$ is closed.\end{proof}

Consider quotients of DSH algebras. Suppose that $A$ is a DSH algebra and $\psi: A \to B$ is a surjective homomorphism. We define $\widehat{\psi}$ to be the injective, single-valued map from $\widehat{B} \to \widehat{A}$, the spectra of equivalence classes of irreducible representations of $B$ and $A$. For each $i$, we have that the subset of $\widehat{A}$ of representations of dimension $n_i$ is homeomorphic to $X_i\setminus Y_i$. Then for $i=1,\ldots,l$,  define $X_i^\prime$ to be the closure of $X_i \bigcap \widehat{\psi}(\widehat{B})$, and $Y_i^\prime$ to be $X_i^\prime \bigcap Y_i$. Since $\widehat{B}$ is compact, it follows that $\widehat{\psi}(\widehat{B})$ is closed, and therefore if $y\in Y_i^\prime$, each corresponding point $x_1,\ldots,x_t$ in the diagonal decomposition of $\phi_{i-1}$ lies in some $X_j^\prime$ for $j<i$. Define $A(1)^\prime$ to be $C(X_1^\prime, M_{n_1})$, and define a diagonal homomorphism $\phi_1^\prime: A(1)^\prime \to C(Y_2^\prime, M_{n_2})$ using the same diagonal decomposition as $\phi_1$ on $Y_i$ restricted to $Y_i^\prime$. Define $A(2)^\prime$ to be $ A(1)^\prime \oplus_{C(Y_2^\prime, M_{n_2})} C(X_2^\prime, M_{n_2})$. Similarly, for $i=2,\ldots,l-1$, we define $\phi_i^\prime: A(i)^\prime \to C(Y_{i+1}^\prime, M_{n_{i+1}})$ via the restriction of the diagonal decomposition of $\phi_i$ to $Y_{i+1}^\prime$ and $A(i+1)^\prime$ as $ A(i)^\prime \oplus_{C(Y_{i+1}^\prime, M_{n_{i+1}})} C(X_{i+1}^\prime, M_{n_{i+1}})$.

\begin{lemma} Every quotient of a DSH algebra is DSH.\label{quotient}\end{lemma}
\begin{proof}We prove that the DSH algebra $A(l)^\prime$ is isomorphic to $B$. For each $\pi\in \widehat{B}$, we choose a unique representation $R[\pi]$ from the equivalence class $\pi$ such that for every $f\in B$, $R[\pi](f)= \tilde{f}_i(\widehat{\psi}(\pi))$ where $\tilde{f}$ is any preimage of $f$ and $\widehat{\psi}(\pi)\in X_i$. Then for each element of $f\in B$ we can define an element $\theta(f)\in A(l)^\prime$ by $\theta(f)_i(x)= R[\widehat{\psi}^{-1}(x)](f)$. We can then check that this is a homomorphism and a bijection.\end{proof}

When discussing a sequence $\mathrm{lim}(A_j,\phi_j)$ of DSH algebras, we will identify the spaces in the composition sequence of $A_j$ as $X_i^j$ and the dimensions as $n_i^j$ with subspaces $Y_i^k$ and $B_{i,k}^j$. 

We say that a map $\phi_{j_2,j_1}:A_{j_1}\to A_{j_2}$  between DSH algebras is diagonal if for every $x\in X_{i}^{j_2}$ there are points $x_1,\ldots,x_t$, where $x_k\in X_{i_k}^{j_1}$ such that for every $f\in A_{j_1}$, $\phi_{j_2,j_1}(f)_i(x)=\mathrm{diag}(f_{i_1}(x_1),\ldots,f_{i_t}(x_t))$.

Dynamical systems yield the motivating examples of DSH algebras. Let $X$ be an infinite compact Hausdorff space and let $\sigma: X \to X$ be a minimal homeomorphism. We denote also by $\sigma$ the automorphism of $C(X)$ given by $\sigma(f)= f \circ \sigma^{-1}$. We let $u$ be a unitary such that $ufu^*=\sigma(f)$.

Let $Y\subseteq X$ be the closure of an open set in $X$ and consider the algebra  $A_Y$ generated by $ \{ f, ug: g(x)\vert_Y=0   \}$. By unpublished results of Q. Lin, this algebra is subhomogeneous and can be described explicitly as follows.

For $y\in Y$, write $R (y)= \min \{n>0: \sigma^n(y)\in Y\}$. This is the first return time for $y$. By compactness $R(Y)$ is a finite set, and we list its values as $n_1\ldots, n_l$. Moreover, we write $X_i = \overline{R^{-1}(n_i)}$ for $i=1,\ldots, l$. Then $A_Y$ can be written as a subalgebra of $\bigoplus_{i=1}^l M_{n_i}(C(X_i))$.

Define $Y_i=X_i \setminus R^{-1}(n_i)$. Whenever $y\in Y_i$, there exists ${t_1}, \ldots {t_s}$ and $y^\prime \in X_{t_1}$ such that $n_{t_1}+\cdots + n_{t_s} = n_i$ and $\sigma^{n_{t_1} +\cdots +n_{t_j}}(y^\prime)\in Y$ for $1\leq j\leq s$. In this case when $f\in A_Y$, $f_i(y) = \mathrm{diag}( f_{t_1}(y^\prime), f_{t_2}(\sigma^{n_{t_1}}(y^\prime)),\ldots, f_{t_s}(\sigma^{n_{t_1}+\cdots + n_{t_{s-1}}}(y^\prime)))$. The algebra $A_Y$ is a DSH algebra with a composition sequence of length $l$, spaces $X_i$, restriction subspaces $Y_i$, and dimensions $n_i$.

For $x\in X$, we write $A_x$ rather than $A_{\{x\}}$. $A_x$ is simple and, by \cite{archey}, is a centrally large subalgebra of the crossed product. We will show $A_x$ has stable rank one. 

It is clear from the generators that if $Z\subseteq Y$ then $A_Y \subseteq A_Z$ and we let $\phi$ be the natural embedding.
 \begin{lemma} $\phi$ is a diagonal map between DSH algebras.\label{decomp}\end{lemma}
\begin{proof} We will give an explicit diagonal description of $\phi$. We define the return time of a point $z\in Z$ to $Z$ as $R_Z(z)$. $Z$ decomposes into sets $Z_i=\overline{R_Z^{-1}(q_i)}$ where $q_1,\ldots,q_w$ is the list of possible return times. Since every point in $Z_i\setminus R_Z^{-1}(q_i)$ lies in $Z_j$ for some $j<i$, it suffices to describe $\phi(f)$ at an arbitrary point in $R_Z^{-1}(q_i)$.
Suppose $z\in R_Z^{-1}(q_i)$. There exist $n_{t_1}, \ldots n_{t_s}$ such that $n_{t_1}+\cdots + n_{t_s} = q_i$ and when $1\leq k\leq q_i$, $\sigma^k(z)\in Y$ if and only if $k=n_{t_1} +\cdots +n_{t_j}$ for some $1\leq j\leq s$. Let $S_j=n_{t_1} +\cdots +n_{t_j}$ and assume $\sigma^{S_j}(z)\in X_{i_j}$. Then for $f\in A_Y$, we will show
\begin{equation}
\phi(f)_i(z) = \mathrm{diag}( f_{i_0}(z),f_{i_1}(\sigma^{n_{t_1}}(z)),\ldots, f_{i_{s-1}}(\sigma^{S_{s-1}}(z))).\label{embed}
\end{equation}
 Since $\phi$ is evidently a diagonal homomorphism, it suffices to prove \eqref{embed} for the generators $f$ and $ug$. Suppose $f\in C(X)$, and denote by $\iota_Y$ the function on $\bigsqcup X_i$ in our presentation of $A_Y$ and similarly with $\iota_Z$. We have that $\iota_Y(f)= \bigoplus_{i=1}^{l} \mathrm{diag}(f\circ\sigma, \ldots, f\circ\sigma^{n_i})$. Suppose then that for some $j$, $\sigma^j(z)\in X_k$. Then $\iota_Y(f)_k(\sigma^j(z))=\mathrm{diag}(f\circ\sigma^{j+1}, \ldots, f\circ\sigma^{j+n_k})$.
We may decompose the sequence $1,\ldots, q_i$ into $1, \ldots, n_{t_1}$ followed by $S_j+1,\ldots ,S_j+n_{t_{j+1}}$ for $j=1,\ldots, s-1$. For each of these subsequences  
$$\mathrm{diag}(f(\sigma^{S_j+1}(z)), \ldots, f(\sigma^{S_j+n_{t_{j+1}}}(z)))=\iota_Y(f)_{i_j}(\sigma^{S_j}(z)).$$
We then have that 
$$\phi(\iota_Y(f))_{i}(z) = \mathrm{diag}( \iota_Y(f)_{i_0}(z),\ldots, \iota_Y(f)_{i_{s-1}}(\sigma^{S_{s-1}}(z)))$$
\begin{align*}	
&= \mathrm{diag}( f\circ\sigma(z), \ldots, f\circ\sigma^{n_{t_1}}(z),\ldots, f\circ\sigma^{(S_{s-1}+1}(z), \ldots, f\circ\sigma^{S_s}(z))\\
&=\mathrm{diag}(f\circ\sigma(z), \ldots, f\circ\sigma^{q_i}(z))\\
&= \iota_Z(f)_i(z).
	\end{align*}

Suppose $g\in C(X)$ and $g(y)=0$ for all $y\in Y$. Then $ug\in A_Y\subseteq A_Z$ and we denote its two presentations by $\iota_Y(ug)$ and $\iota_Z(ug)$. For $z\in Z_i$,

$$  \iota_Z(ug)_i(z)=\begin{bmatrix}
	0 & & & \\
	g\circ\sigma(z) & 0 & & \\
	& \ddots & \ddots & \\
	& & g\circ\sigma^{q_i-1}(z) & 0\\
\end{bmatrix}.$$

Moreover, since $\sigma^{S_j}(z)\in X_{i_j}$ for $0\leq j\leq s$, $ug\circ\sigma^{S_j}(z)=0$. Thus $\iota_Z(ug)_i(z)$ can be decomposed into $s$ diagonal blocks of size $n_{t_1},\ldots, n_{t_{s}}$. Each such block can be identified with a value of $\iota_Y(ug)_{i_j}(\sigma^{S_j}(z))$ in our presentation of $A_Y$ via the equation

$$\iota_Y(ug)_{i_j}(\sigma^{S_j}(z)) =
\begin{bmatrix}
	0 & & & \\
	g\circ\sigma^{{S_j}+1}(z) & 0 & & \\
	& \ddots & \ddots & \\
	& & g\circ\sigma^{S_{j+1}-1}(z) & 0\\
\end{bmatrix}. $$
Hence $\mathrm{diag}(\iota_Y(ug)_{i_0}(z),\ldots,\iota_Y(ug)_{i_{s-1}}(\sigma^{S_{s-1}}z))$ $=\iota_Z(ug)_k(z)= \phi(\iota_Y(ug))_k(z)$.\end{proof}

\section{Permutation unitaries}

As in \cite{ho}, given a permutation $\sigma \in S_n$, denote by $U[\sigma]$ the permutation matrix corresponding to $\sigma$. When $\sigma$ is a transposition, let $u_\sigma: [0,1] \to \mathcal{U}(M_n)$ be a continuous map such that

\begin{enumerate}
\item $u_\sigma(0)=1_n$,
\item $u_\sigma(1)= U[\sigma]$,
\item if either $i$ or $j$ is fixed by $\sigma$, then $u_\sigma(x)_{i,j} = \delta_{i,j}$ for all $x$,
\end{enumerate}
where $\delta_{i,j}$ is the Kronecker delta function.

We may choose $u_\sigma$ such that for every $t$, if $B=u_\sigma(t) A u_\sigma(t)^*$, then $B_{i,j}$ is a linear combination of the entries $A_{i,j},A_{\sigma(i),j},A_{i,\sigma(j)}$, and $A_{\sigma(i),\sigma(j)}$.

Indeed, we can choose functions $g_1, g_2, g_3, g_4: [0,1]\to \mathbb{C}$ such that for any $k_1<k_2$, $u_{(k_1\;k_2)}(t)_{i,j}= \delta_{i,j}$ for all $t$ if either $i$ or $j$ is not in $\{k_1,k_2\}$ and $u_{(k_1\;k_2)}(t)_{k_1,k_1}=g_1(t)$,$u_{(k_1\;k_2)}(t)_{k_1,k_2}=g_2(t)$,$u_{(k_1\;k_2)}(t)_{k_2,k_1}=g_3(t)$,$u_{(k_1\;k_2)}(t)_{k_2,k_2}=g_4(t)$. By defining $u$ in this way we have that whenever $k_3>k_2>k_1$, we have 
\begin{equation}U[(k_2\; k_3)] u_{(k_1\;k_2)}(t)U[(k_2\; k_3)]= u_{(k_1\;k_3)}(t).\label{conj}\end{equation}

\begin{dfn}We say an $n\times n$ matrix $A$ has a {\em zero cross} at position $k$ if $A_{k,j}=0$ and $A_{i,k}=0$ for all $i,j$.\end{dfn}

\begin{lemma} Suppose that $A$ is an $n \times n$ matrix with zero crosses at distinct positions $z_1,\ldots, z_m$. Let $t_1,\ldots,t_m \in [0,1]$.

For some $k\in\{1,\ldots,n\}\setminus\{z_1,\ldots,z_m\}$, consider the matrix $$B=u_{(k\;z_m)}(t_m) \cdots u_{(k\;z_1)}(t_1) A u_{(k\;z_1)}(t_1)^* \cdots u_{(k\;z_m)}(t_m)^*.$$
\begin{enumerate}
\item If $i$ and $j$ are not in $\{ k, z_1,\ldots, z_m\}$ then $B_{i,j}=A_{i,j}$.
\item If $i$ is in $\{ k, z_1,\ldots, z_m\}$ and $j$ is not then $B_{i,j}$ is non-zero only if $A_{k,j}$ is non-zero.
\item If $j$ is in $\{ k, z_1,\ldots, z_m\}$ and $i$ is not then $B_{i,j}$ is non-zero only if $A_{i,k}$ is non-zero.
\item If any of $t_1,\ldots,t_m =1$ then $B$ has a zero cross at position $k$.
\end{enumerate}\label{permute}\end{lemma}
\begin{proof}We do this by induction on $m$. Consider $u=u_{(k\;z)}(t)$ where $A$ has a zero cross at position $z$. If $i$ and $j$ are not in $\{ k,z\}$, then for $B= u A u^*$,
\begin{align*}
B_{i,j} &= \sum_{x=1}^n \sum_{y=1}^n u_{i,x} A_{x,y} u_{y,j}^*& \\
&= \sum_{x=1}^n \sum_{y=1}^n \delta_{i,x} A_{x,y} \delta_{y,j}\\
&= A_{i,j}.\\
\end{align*}
If $i\in \{k,m\}$ and $j$ is not, then 
\begin{align*}
B_{i,j} &= \sum_{x=1}^n \sum_{y=1}^n u_{i,x} A_{x,y} u_{y,j}^* & \\
& =u_{i,k}A_{k,j} + u_{i,z}A_{z,j}\\
&= u_{i,k}A_{k,j}\\
\end{align*}
since $A_{z,j}$ is zero. This establishes (2). The case wherein $j\in \{k,m\}$ and $i$ is not is similar. Finally, when $t=1$, $B_{i,j} = A_{\sigma(i),\sigma(j)}$. Hence we have $B_{i,j}=0$ whenever $i$ or $j$ is equal to $k$.

Suppose all four claims hold for a given $m$. Suppose $A$ has a zero cross at position $z_{m+1}$. We first show that  $B=u_{(k\;z_m)}(t_m) \cdots u_{(k\;z_1)}(t_1) A u_{(k\;z_1)}(t_1)^* \cdots u_{(k\;z_m)}(t_m)^*$ also has a zero cross at position $z_{m+1}$. By (1), $B_{i,z_{m+1}} = A_{i,z_{m+1}}=0$ whenever $i\notin \{ k, z_1,\ldots, z_m\}$. By (2), when $i\in \{ k, z_1,\ldots, z_m\}$, $B_{i,z_{m+1}} = 0$ because $A_{k,z_{m+1}} = 0$. Hence $B_{i,z_{m+1}} = 0$ for all $i$ and similarly $B_{z_{m+1},j} = 0$. 
Consider $B^\prime = u_{(k\;z_m)}(t_m)B u_{(k\;z_m)}(t_m)^*$. If $i,j \notin \{ k, z_1,\ldots, z_{m+1}\}$ then, as above, $B^\prime_{i,j}=B_{i,j}$. By assumption this is the same as $A_{i,j}$. Suppose $i$ is in $\{ k, z_1,\ldots, z_{m+1}\}$ and $j$ is not. If $i\in \{k,z_{m+1}\}$, then $B^\prime_{i,j}$ is non-zero only if $B_{i,k}$ is non-zero. If $i$ is in $\{ z_1,\ldots, z_{m}\}$, then $B^\prime_{i,j}=B_{i,j}$. In either case, $B^\prime_{i,j}$ is non-zero only if $A_{i,k}$ is non-zero. Claim (3) is similar.

Suppose at least one value in $\{t_1, \ldots, t_{m+1}\}$ is one. It is immediate that if $t_{m+1}=1$ then $B^\prime$ has a zero cross in position $k$. If any of $t_1, \ldots t_m$ is $1$, then by assumption $B$ has a zero cross at position $k$ as well as at $z_{m+1}$. It follows that $B^\prime$ has a zero cross in position $k$. \end{proof}

\begin{lemma}Suppose we have an $n\times n$ matrix $A$ and $\Delta\in [0,1]^n$ such that $\Delta_i>0$ only if $A$ has a zero cross at position $i$. Suppose that for some $k$, $M\leq n-k+1$, we have that $\Delta_i=1$ for some $i\in \{k,\ldots,k+M-1\}$. Then for $$V=u_{(k\;k+1)}(\Delta_{k+1}) \cdots u_{(k\;k+M-1)}(\Delta_{k+M-1}),$$
$V F V^*$ has a zero cross at position $k$.\label{block1}\end{lemma}
\begin{proof} If $\Delta_j=0$ then $u_{(k\; j)}(\Delta_j)$ is the identity, thus we may rewrite 

\begin{align*}V&= u_{(k\;k+1)}(\Delta_{k+1}) \cdots u_{(k\;k+M-1)}(\Delta_{k+M-1})\\
&= u_{(k\;z_1)}(\Delta_{z_1}) \cdots u_{(k\;z_m)}(\Delta_{z_m})\\
\end{align*}
where $\Delta_{z_j}>0$ for each $j$ and thus $A$ has a zero cross in position $z_j$. If $\Delta_i=0$ for all $i>k$ then this product is vacuous. If this is the case then $\Delta_k=1$ and $A=VAV^*$ already has a zero cross in position $k$. Otherwise, apply Lemma~\ref{permute} (4) to conclude that $VAV^*$ has a zero cross in position $k$.\end{proof}

If $A_{i,j}=0$ whenever $\left|i-j\right|\geq r$, let us say $A$ has diagonal radius $r$. Denote the smallest integer such that this holds as $r(A)$.

\begin{lemma}Suppose $\Delta_i=1$ for some $i\in \{k,\ldots,k+M-1\}$ for $k=k_1, \ldots, k_N$  where $k_{j+1}-k_j\geq M$ and $k_N\leq n-M+1$. Set $$V_i=u_{(k_i\;k_i+1)}(\Delta_{k_i+1}) \cdots u_{(k_i\;k_i+M-1)}(\Delta_{k_i+M-1}).$$

For $j=0,\dots, N$, set $A(j)=V_{j}\cdots V_1 A V_1^* \cdots V_j^*$. Then $A(N)$ has zero crosses at $k_1, \ldots, k_N$ and $r(A(N))\leq r(A)+M-1$.\label{block2}\end{lemma}
\begin{proof}We now show that $A(N)$ has zero crosses at $k_1, \ldots, k_N$. First note that if $B$ has a zero cross at position $z$, and $z\notin \{k_i,\ldots,k_i+M-1\}$, then $V_iBV_i^*$ also has a zero cross at position $z$. This follows from Lemma~\ref{permute} (1),(2), and (3). For each $j$, define the set $K_j= \{k_j,\ldots,k_j+M-1\}$ and observe that since $k_{j+1}-k_j\geq M$, these sets are disjoint. Therefore, since $A$ has a zero cross at position $i$ for some $i\in K_j$, this is also true of $V_{j-1}\cdots V_1 F V_1^* \cdots V_{j-1}^*$. Therefore we may apply the previous lemma and conclude that $V_{j}\cdots V_1 F V_1^* \cdots V_{j}^*$ has zero crosses at positions $k_1,\ldots ,k_j$. Hence $A(N)$ has the desired property.

To bound the diagonal radius of $A(N)$, for each $s,t\in \{1,\ldots,n\}$ consider $A(j)_{s,t}$. If neither $s$ nor $t$ lies in $K_{j+1}$, then $A(j+1)_{s,t}=A(j)_{s,t}$. Since each of $s$ and $t$ lies in at most one set $K_j$, the sequence $A(1)_{s,t}$,\ldots, $A(N)_{s,t}$ takes on at most three values. We will show that if $\left|s-t\right|\geq r(A)+M$ then $A(N)_{s,t}=0$. 
When $\left|s-t\right|\geq r(A)+M\geq M$, $s$ and $t$ cannot lie in the same set $K_j$. Suppose that $s\in K_{j_s}$ and $t\in K_{j_t}$. $A(j_s)_{s,t}$ only differs from $A(j_{s-1})_{s,t}$ if $A(j_{s-1})$ has a zero cross at position $s$ or if $s=k_{j_s}$. In either case, $A(j_s)_{s,t}$ is non-zero only if $A(j_{s-1})_{k_{j_s},t}$ is also non-zero. It follows, then, that if $A_{s,t}, A_{k_s,t}, A_{s,k_t}$ and $A_{k_s,k_t}$ are all zero then $A(N)_{s,t}$ is also zero. If $s$ and $t$ differ by at least $r(A)+M-1$ then each of these co-ordinate pairs differs by at least $r(A)$, and hence all four terms are zero. The cases in which either $s$ or $t$ or both are not in any set $K_j$ are similar.\end{proof}

\begin{lemma}Suppose $A\in M_n$ has zero crosses at positions $z_1, \ldots, z_m$. Then there exists a unitary $V\in C([0,1],M_n)$ such that $V A V^*(1)$ has zero crosses at positions $1,\ldots, m$, $V(0)$ is the identity and for all $\theta\in [0,1]$, $r(V AV^*(\theta)) \leq r(A)+2$.\label{condense}\end{lemma}
\begin{proof}For $0< i\leq j$, let $\delta_j^i: [0,1]\to [0,1]$ be a continuous increasing function with $\delta_j^i(\theta)=0$ for $\theta\leq\frac{i-1}{j}$ and $\delta_j^i(\theta)=1$ for $\theta\geq\frac{i}{j}$. 

Construct $u_j^i\in C([0,1],M_n)$ by 
$$u_j^i(\theta)=   u_{(i\;i+1)}\circ\delta_{j-i}^{j-i}(\theta)\cdots u_{(j-1\; j)}\circ\delta_{j-i}^{1}(\theta).$$
For any $\theta$, there exists at most one value $k$ such that $\delta_{j-i}^k(\theta)\in (0,1)$. Hence 
\begin{align*}
u_j^i(\theta)&=  u_{(j-k\; j-k-1)}(\delta_{j-i}^{k}(\theta))u_{(j-k+1\;j-k)}(1)\cdots u_{(j\; j-1)}(1)   \\
&=  u_{(j-k\; j-k-1)}(\delta_{j-i}^{k}(\theta))U[(j-k\;j-k+1  \cdots j)].\\
\end{align*}
Let us show that for any $A\in M_n$ with a zero cross at position $j$, $r(u_j^{i}(1)A u_j^{i}(1)^*)\leq r(A)+1$, and when $i=1$, $r(u_j^1(1)A u_j^{1}(1)^*)\leq r(A)$. 

Consider $B=u_j^{1}(1)A u_j^{1}(1)^*=U[(1\; 2  \cdots j)]A U[(1\; 2  \cdots j)]^*$. If $A_{s,t}$ is non-zero then neither $s$ nor $t$ is equal to $j$, and $A_{s,t}$ sits in one of four regions, each of which is either fixed (if $s,t>j$), shifted parallel to the diagonal (if $s,t < j$), or shifted toward the diagonal (if $s<j<t$ or $t<j<s)$. No non-zero entry in $A$ is shifted away from the diagonal; hence $r(B)\leq r(A)$.

The case where $i>1$ is similar except that if $i\leq s<j$ and $t<i$ or $i\leq t<j$ and $s<i$ then this possible non-zero entry is shifted by one away from the diagonal. Hence $r(u_j^{i}(1)A u_j^{i}(1)^*)\leq r(A)+1$.

For a fixed $k$ we let $B=U[(j-k\;j-k+1  \cdots j)]AU[(j-k\;j-k+1  \cdots j)]^* $ and we have just shown that $r(B)\leq r(A)+1$. We can rewrite $u_j^{i}(\theta)A u_j^{i}(\theta)^*$ as $$u_{(j-k\; j-k-1)}(\delta_{j-i}^{k}(\theta))B u_{(j-k\; j-k-1)}(\delta_{j-i}^{k}(\theta))^*$$
Let $C=u_{(j-k\; j-k-1)}(\delta_{j-i}^{k}(\theta))B u_{(j-k\; j-k-1)}(\delta_{j-i}^{k}(\theta))$. It then follows from Lemma \ref{block2} with $N=1$ and $M=2$ that $r(C)\leq r(B)+1$. Hence $r(u_j^i(\theta)A u_j^{i*}(\theta))\leq r(A)+2$.

Suppose then that we have $A\in M_n$ with zero crosses at positions $z_1, \ldots, z_m$. Set $V = u_{z_1}^1 \circ \delta_m^1 \cdots u_{z_m}^1 \circ \delta_m^m$. For each $t$, there exists a $k\in \{1,\ldots, m\}$ such that 

$$V(\theta)= u_{z_1}^1(1) \cdots u_{z_{k-1}}^1(1) u_{z_{k}}^1( \delta_m^k(\theta)).$$

Then 

$$r(u_{z_{k-1}}^1(1)\cdots u_{z_1}^1(1)  A u_{z_1}^1(1)^* \cdots u_{z_{k-1}}^1(1)^*)=r(A)$$
and after conjugating by $u_{z_{k}}^1( \delta_m^k(\theta))$ and applying the above estimate, we have $$r(VAV^*(\theta))\leq r(A)+2.$$\end{proof}

Fix $N<n$. For any $k\leq n$, denote by $\gamma_{k,n}$ the cycle $(k\; k+1\cdots  n)$. For any $k$ such that $N\leq k \leq n$, denote by $\eta_{k,n}$ the product of transpositions $(k-N+1\; n-N+1)\cdots (k\; n)$, which exchanges the $N$ entries up to and including $k$ with the final $N$ entries in $\{1,\ldots ,n\}$. We will use $u_{\eta_{k,n}}(\theta)$ to denote $u_{(k-N+1\; n-N+1)}(\theta)\cdots u_{(k\; n)}(\theta)$. Note that whenever $k\leq n-N$ these terms all commute and if $\theta=0$ then they are all trivial.

\begin{lemma}If $k\leq i-N$ and $i\leq n-N$ then 
\begin{equation}u_{\eta_{n,k}}(\theta)= U[\eta_{i,n}] u_{\eta_{i,k}}(\theta) U[\eta_{i,n}].\label{fullconj}\end{equation}\end{lemma}
\begin{proof} By definition 
	$$U[\eta_{i,n}] u_{\eta_{i,k}}(\theta) =\Big(\prod_{j=-(N-1)}^0 U[(i+j\; n+j)]\Big)\Big(\prod_{j=-(N-1)}^0 u_{\eta_{k+j \; i+j}}(\theta)\Big).$$
	Since $i\leq n$, $i+j$ cannot equal $n+j^\prime$ unless $j^\prime\leq j$. Since $k\leq i-N$, $k+j$ is never equal to $i+j^\prime$ or $n+j^\prime$ for any $j^\prime\in \{-(N-1),\ldots,0\}$. It follows that $u_{\eta_{(k+j\; i+j)}}$ commutes with every $U[(i+j^\prime\; n+j^\prime)]$ such that $j^\prime>j$. Hence 
	$$U[\eta_{i,n}] u_{\eta_{i,k}}(\theta) =\prod_{j=-(N-1)}^0 U[(i+j\; n+j)]u_{\eta_{k+j \; i+j}}(\theta)$$
	Since $i\leq n-N$, all the terms $U[(i+j\; n+j)]$ commute with one another. We therefore have 
	$$U[\eta_{i,n}] u_{\eta_{i,k}}(\theta)U[\eta_{i,n}] =\prod_{j=-(N-1)}^0 U[(i+j\; n+j)]u_{\eta_{k+j \; i+j}}(\theta)U[(i+j\; n+j)].$$
	Then by Eq.\eqref{conj} we have 
	$$U[\eta_{i,n}] u_{\eta_{i,k}}(\theta)U[\eta_{i,n}] =\prod_{j=-(N-1)}^0 u_{\eta_{k+j \; n+j}}(\theta)=u_{\eta_{n,k}}(\theta).$$\end{proof}

The following calculation is elementary.
\begin{lemma}When $i\leq n-N$, we have
$$U[\gamma_{1,n}]^N U[\eta_{i-1,n}]=U[\gamma_{1,i-1}]^N U[\gamma_{i,n}]^N.$$
\end{lemma}
We can now construct the final class of unitaries we require.
\begin{lemma}For some $n$, let $\Theta\in [0,1]^{(n)}$ and suppose the first entry is $1$, the final $N$ entries are all  zero and at most one of any consecutive $N$ entries is non-zero. We will write $\Theta_k$ to denote the $k$th entry of $\Theta$. Consider the unitary
$$ V_n(\Theta)= U[\gamma_{1,n}]^N\Big(\prod_{k=N}^{n-1} u_{\eta_{k,n}}(\Theta_{k+1}) \Big).$$
Suppose $\Theta_k =1$ for $k\in K=\{k_1,\ldots,k_m\}$ where $k_1=1$ and let $k_{m+1}=n+1$. Then
\begin{equation}\label{Vn}
	V_n(\Theta)= \mathrm{diag}( V_{k_2-k_1} (\Theta_{k_1},\ldots,\Theta_{k_2-1}),\ldots, V_{k_{m+1}-k_{m}}(\Theta_{k_{m}},\ldots,\Theta_{k_{m+1}-1})).\end{equation}\label{V}\end{lemma}

\begin{proof}If $\Theta_i=1$ then $i\leq n-N$. Note that for $k\in 1,\ldots,i-2$, $u_{\eta_{k,n}}(\Theta_{k+1})= U[\eta_{i,n}] u_{\eta_{k,i}}(\Theta_{k+1})U[\eta_{i,n}]$. This holds by Eq.~\eqref{fullconj} for $k\leq i-N$, and holds because $\Theta_{k+1}=0$ for $i-N\leq k<i-1$. We can then compute that 
$$\prod_{k=N}^{i-2} u_{\eta_{k,n}}(\Theta_{k+1}) =\prod_{k=N}^{i-2}U[\eta_{i-1,n}] u_{\eta_{k,i-1}}(\Theta_{k+1}) U[\eta_{i-1,n}]$$
$$=U[\eta_{i-1,n}]\Big(\prod_{k=N}^{i-2}  u_{\eta_{k,i-1}}(\Theta_{k+1})\Big) U[\eta_{i-1,n}]. $$
We multiply on the left by $U[\gamma_{1,n}]^N$ and apply the previous lemma.
$$U[\gamma_{1,n}]^N\Big(\prod_{k=N}^{i-2} u_{\eta_{k,n}}(\Theta_{k+1}) \Big)=U[\gamma_{1,i-1}]^N U[\gamma_{i,n}]^N\Big(\prod_{k=N}^{i-2}  u_{\eta_{k,i-1}}(\Theta_{k+1})\Big)U[\eta_{i-1,n}]$$
$$U[\gamma_{1,n}]^N\Big(\prod_{k=N}^{i-2} u_{\eta_{k,n}}(\Theta_{k+1}) \Big)=U[\gamma_{1,i-1}]^N \Big(\prod_{k=N}^{i-2}  u_{\eta_{k,i-1}}(\Theta_{k+1})\Big)U[\gamma_{i,n}]^N U[\eta_{i-1,n}]$$
Let $i=k_2$ and multiply on the right by $U[\eta_{k_2-1,n}]$.
$$ U[\gamma_{1,n}]^N\Big(\prod_{k=N}^{k_2-2} u_{\eta_{k,n}}(\Theta_{k+1})\Big)U[\eta_{k_2-1,n}]$$
$$=U[\gamma_{1,k_2-1}]^N\Big(\prod_{k=N}^{k_2-2} u_{\eta_{k,k_2-1}}(\Theta_{k+1})\Big) U[\gamma_{k_2,n}]^N U[\eta_{k_2-1,n}]U[\eta_{k_2-1,n}]$$
$$=U[\gamma_{1,k_2-1}]^N\Big(\prod_{k=N}^{k_2-2} u_{\eta_{k,k_2-1}}(\Theta_{k+1})\Big) U[\gamma_{k_2,n}]^N$$
Since $\Theta_{k_2}=1$, we can rewrite $V_n$.
$$V_n(\Theta)=U[\gamma_{1,n}]^N\Big(\prod_{k=1}^{k_2-2} u_{\eta_{k,n}}(\Theta_{k+1})\Big)U[\eta_{k_2-1,n}]\Big(\prod_{k=k_2}^{n-1} u_{\eta_{k,n}}(\Theta_{k+1})\Big)$$
$$=U[\gamma_{1,k_2-1}]^N\Big(\prod_{k=N}^{k_2-2} u_{\eta_{k,k_2-1}}(\Theta_{k_2+1}) \Big)U[\gamma_{k_2,n}]^N \Big(\prod_{k=k_2}^{n-1} u_{\eta_{k,n}}(\Theta_{k+1})\Big)$$
Using the fact that $\Theta_{k+1}=0$ for $k_2\leq k <k_2+N-1$ we have
$$V_n(\Theta)=U[\gamma_{1,k_2-1}]^N  \Big(\prod_{k=N}^{k_2-2} u_{\eta_{k,k_2-1}}(\Theta_{k+1})\Big)  U[\gamma_{k_2,n}]^N  \Big(\prod_{k=k_2+N-1}^{n} u_{\eta_{k,n}}(\Theta_{k+1})\Big) $$
$$=\mathrm{diag}(V_{k_2-k_1}(\Theta_{k_1},\ldots,\Theta_{k_2-1}),V_{k_{m+1}-k_2}(\Theta_{k_2},\ldots,\Theta_{k_{m+1}-1})).$$
We may then repeat this on $V_{k_{m+1}-k_2}(\Theta_{k_2},\ldots,\Theta_{k_{m+1}-1})$ using $k_3$ in place of $k_2$ and Eq.~\ref{Vn} follows.
\end{proof}

\begin{lemma}If $A$ is an $n\times n$ matrix with $r(A)\leq N$, $\Theta$ is as above, and whenever $\Theta_k>0$, $A$ has zero crosses at positions $k,\ldots, k+N-1$, then $AV_n $ is a strictly lower triangular matrix.\label{triangulate}\end{lemma}
\begin{proof} Suppose $A$ is an $n\times n$ matrix with $r(A)\leq N$, and whenever $\Theta_k>0$, $A$ has zero crosses at positions $k,\ldots, k+N-1$. Then $AU[\gamma_{1,n}]^N $ is strictly lower triangular and whenever $\Theta_k>0$, $A U[\gamma_{1,n}]^N $ has zero columns at positions $k-N,\ldots, k-1$ as well as at positions $n-N+1,\ldots, n$. Therefore, for every $k$ such that $\Theta_k>0$, every column that $u_{\eta_{k,i}}(\Theta_{k})$ does not fix is already zero. Therefore $AV_n $ is also strictly lower triangular.\end{proof}

\section{Unitaries in DSH Algebras}

We will need the following approximation result.

\begin{lemma}Suppose that $f\in A$ where $A$ is DSH. Then for any $\epsilon>0$ there exist $f^\prime\in A$ such that $\left|f-f^\prime\right|<\epsilon$ and open sets $U_{i,k}$ such that $B_{i,k}\subseteq U_{i,k}$ and whenever $x\in U_{i,k}$, $f_i^\prime(x)$ has a block point at position $k$. Moreover, $f^\prime$ retains any zero crosses of $f$ and $r(f_i^\prime(x))\leq r(f_i(x))$ for all $i$ and $x\in X_i$.\label{blockpoint} \end{lemma}
\begin{proof} Let $\epsilon>0$. We construct $f^\prime$ by simultaneously modifying $f$ on every $X_i$. We do this by considering each function $f_i(x)_{s,t}$ for indices $s,t\in \{1,\ldots n_i\}$ as a function in $C(X_i)$. For $\delta>0$ let $g_\delta\in C(\mathbb{C})$ be given by $g_\delta(z)= z{\mathrm{max}(0, \left|z\right| - \delta)}/{\left|z\right|} $. For any space $X$ and any $f\in C(X)$ and any $x\in X$ such that $f(x)=0$, $g_\delta \circ f$ is zero on an open set containing $x$. Since each space $X_i$ is compact, we may choose a $\delta$ such that $\left|f_i(x)_{s,t}- g_\delta(f_i(x)_{s,t})\right|<{\epsilon}/{n_l^2}$, where $n_l$ is the largest dimension of any irreducible representation of $A$.

For $x\in X_i$, we define $f_i^\prime(x)$ to be the matrix given by the entries $f_i^\prime(x)_{s,t}=g_\delta(f_i(x)_{s,t})$. Then for each $x$, $\left|f_i^\prime(x)-f_i(x)\right|<\epsilon$, hence $\left|f^\prime -f\right|<\epsilon$. Since we chose a single value of $\delta$ and modified every entry in a uniform manner, it is easy to see that when $y\in Y_i$, $f_i^\prime(y)=\mathrm{diag}( f_{i_1}^\prime(x_{y_1}),\ldots,f_{i_s}^\prime (x_{y_s}) )$, hence $f^\prime \in A$.

For every $i,x,s,t$, if $f_i(x)_{s,t}=0$ then $f_i^\prime(x)_{s,t}=0$. Therefore $f_i^\prime$ has a zero cross wherever $f_i$ does and $r(f_i^\prime(x))\leq r(f_i(x))$. For $k\in \{1,\ldots, n_i\}$.

Fix $i$. Define $P(k)\subseteq \{1,\ldots, n_i\}^2$ by $P(k)= \{(s,t): s<k\mbox{ and }t\geq k \mbox{ or } t<k\mbox{ and }s\geq k\}$. For any $k$, when $x\in B_{i,k}$, $f_i(x)_{s,t}=0$ whenever $(s,t)\in P(k)$. For each $(s,t)\in P(k)$, there is an open set $O_{s,t}$ containing $B_{i,k}$ on which $f_i^\prime(x)_{s,t}=0$. Taking $U_{i,k}$ to be the intersection $\bigcap_{(s,t)\in P(k)} O_{s,t}$, $f_i^\prime(x)$ has a block point at position $k$ for all $x\in U_{i,k}$.\end{proof}

We will need to construct certain indicator functions in order to implement the unitaries defined above. %In so doing, we will often need to ensure that a function $\Theta_k$ in $C(X_i, [0,1])$ is zero on some closed set $f_{i,k}$ that is separated from $B_{i,k}$. Our construction leads us to the following notion.

%\begin{dfn} Suppose we have a DSH algebra $A\subset \bigoplus_i C(X_i, M_{n_i})$ and a family of closed sets $f_{i,k}$ where $1\leq k \leq n_i$.  Consider $y\in f_{i,k}\cap Y_i$ with corresponding points $x_{y,1}, \ldots x_{y,s}$, where $x_{y,t}\in X_{i_t}$, and suppose $k = n_{i_1}+\cdots+ n_{i_t} +k^\prime$ where $k^\prime\leq n_{i_{t+1}}$. If it is the case that for every such $i,y,k$, $x_{y,t+1}\in f_{i_{t+1}, k^\prime}$, then we call the family of sets $f_{i,k}$ {\em well-constructed}.\end{dfn}

\begin{lemma} Suppose that $A$ is a DSH algebra, $M\in\mathbb{N}$, $K=\{K_1,\ldots,K_m\}$ is a set of integers such that $K_{t+1}-K_t \geq M$, $K_1\geq 0$, and $K_m\leq n_1-M$. For each $i,k$ let $F_{i,k}\subset X_i$ be a closed set separated from $B_{i,k-K_t}$ for each $t\in\{1,\ldots,m\}$. Then there is a function $\Theta\in A$ such that  
\begin{enumerate}
\item for every $x\in X_i$, $\Theta_i(x)$ is a diagonal matrix with entries in $[0,1]$.
\item For any $M$ consecutive entries on the diagonal, at most one is non-zero.
\item $\Theta_i(x)_{k,k}=0$ for all $k\in \{n_i-(M-1),\ldots, n_i\}$.
\item For each $i,k$, $\Theta_i(x)_{k,k}=0$ for all $x\in F_{i,k}$.
\item For each $i,k$, there is an open subset $U_{i,k}\subset X_i$ such that $B_{i,k}\subseteq U_{i,k}$ and if $x\in U_{i,k}$ then $\Theta_i (x)_{k+K_t,k+K_t}=1$ for all $t\in\{1,\ldots,m\}$.

\end{enumerate}

\label{indicator}\end{lemma}
\begin{proof} We claim it is sufficient to find $\Theta\in A$ for which conditions (1), (2), and (3) hold and furthermore $\Theta_i (x)_{k,k}=1$ if and only if there exists a $j$ and a $t\in\{1,\ldots,m\}$ such that $k=j+K_t$ and $x\in B_{i,j}$. We call this condition (6). Given such a $\Theta$, we will modify it to satisfy conditions (4) and (5). Let $g:[0,1]\to [0,1]$ be any function for which $g(x)=0$ on some neighbourhood of $0$ and $g(x)=1$ on some neighbourhood of $1$.

On each $X_i$ and for each $k$, we consider $\Theta_i(x)_{k,k}$ to be a function in $C(X_i,[0,1])$, and we refer to this function as $\Theta_{i,k}$. Consider $\Theta_i(x)= \mathrm{diag}( \Theta_{i,1}(x),\ldots, \Theta_{i,n_i}(x))$. We define $\Theta_i^\prime\in C(X_i, M_{n_i})$ by $\Theta_i^\prime(x)= \mathrm{diag}( g\circ\Theta_{i,1}(x),\ldots, g\circ\Theta_{i,n_i}(x))$. Letting $\Theta^\prime = \oplus_i \Theta_i^\prime$, we may check that this lies in $A$. It is clear that $\Theta^\prime$ satisfies condition (1), and since $g$ preserves $0$, it also satisfies (2), and (3).

Since each $F_{i,k}$ is separated from $B_{i,k-K_t}$ for each $t\in\{1,\ldots,m\}$, condition (6) tells us that $\Theta_{i,k}(x)<1$ whenever $x\in F_{i,k}$. Since each set $F_{i,k}$ is compact, there exists a maximum value $\delta<1$ such that for every $i,k$, $\Theta_{i,k}(x)\leq\delta$ whenever $x\in F_{i,k}$. Therefore we may choose $g$ such that $g(x)=0$ on $[0,\delta]$, so that $g\circ \Theta_{i,k}(x)=0$ whenever $x\in F_{i,k}$. Finally, it follows from the definition of $g$ that for each any $j$ and $k\in K$, $\Theta_i (x)_{j+k,j+k}=1$ on an open set $O_{k}$ containing $B_{i,j}$. We take $U_{i,k}$ to be the intersection $\bigcap_{k\in K} O_k$, and $\Theta^\prime$ therefore satisfies the lemma, proving our claim.

To construct $\Theta$, we proceed recursively on $i$.

For $i=1$, we simply let $\Theta_1=\mathrm{diag}(\chi_K(1), \chi_K(2),\ldots, \chi_K(n_1))$ for all $x$ where $\chi_K$ is the characteristic function for $K$. Suppose we have already defined $\Theta_{i^\prime}$ for $i^\prime <i$ and that for each $k\leq n_1$, $\Theta_{i^\prime,k}$ is the constant function $\chi_K(k)$ and for $k\geq n_{i^\prime}-n_1+1$, $\Theta_{i^\prime,k}=0$. 

Then $\Theta_i$ is determined on $Y_i$. Fixing $y\in Y_i$, we check that conditions (1)-(3) and (6) are satisfied by $\Theta_i(y)=\mathrm{diag}( \Theta_{i_1}(x_{1}),\ldots,\Theta_{i_s} (x_{s}) )$. Condition (1) is trivial. (2) needs only to be checked when we consider $M$ consecutive entries spanning distinct components $\Theta_{i_t} (x_{y_t})$. Since $M<n_1$, it suffice to consider two consecutive blocks $\Theta_{i_t} (x_{y_t})\oplus \Theta_{i_{t+1}} (x_{y_{t+1}})$. Condition (3) applied to $\Theta_{i_t} (x_{y_t})$ gives us condition (2).  Condition (3) for $\Theta_i(y)$ follows from the same for $\Theta_{i_s} (x_{y_s})$. 
To check condition (6), fix $k$ and consider $y\in  Y_i$ with corresponding points $x_{y,1}, \ldots x_{y,s}$, where $x_{y,t}\in X_{i_t}$, and suppose $k = n_{i_1}+\cdots+ n_{i_t} +k^\prime$ where $k^\prime\leq n_{i_{t+1}}$. Then $\Theta_{i,k}(y)=\Theta_{i_{t+1},k^\prime}(x_{t+1})$. $y\in B_{i,j}$ if and only if $j\in\{1, n_{i_1}+1, n_{i_1}+n_{i_2}+1,\ldots\}$, and the only value of $j$ such that $y\in B_{i,j}$ for which there can exist a $K_t \in K$ such that $k=j+K_t$ is $j=n_{i_1}+\cdots+ n_{i_t} +1$. Hence there exists a $j$ and a $t\in\{1,\ldots,m\}$ such that $k=j+K_t$ and $x\in B_{i,j}$ if and only if $k^\prime \in K$. We can assume $x_{t+1}$ lies in $X_{i_{t+1}}\setminus Y_{i_{t+1}}$, so $x_{t+1}$ lies in $B_{i_{t+1},j}$ if and only if $j=1$. Therefore, applying condition (6) to $\Theta_{i_{t+1}}$, $\Theta_{i_{t+1},k^\prime}(x_{t+1})=1$ if and only if $k^\prime\in K$. This shows that $\Theta_{i}$ satisfies (6) on $Y_i$.

We extend $\Theta_{i,k}$ to $X_{i}$ as the constant function $\chi_K(k)$ for every $k\leq n_1$ and the final $n_1-1$ functions as the constant zero function. 

Suppose we have defined $\Theta_{i,1},\ldots ,\Theta_{i,k-1}$ on all of $X_{i}$ and that $k$ does not fall into one of the sets for which $\Theta_{i,k}$ has already been given. Assume $\mathrm{supp}(\Theta_{i,k})\subset Y_i$ is disjoint from $\bigcup_{t=1}^{M-1} \overline{ \mathrm{supp}(\Theta_{i,k-t})}$.
We first set $f= \Theta_{i,k}$ on $Y_i$ and $f=0$ on $\bigcup_{t=1}^{M-1} \overline{ \mathrm{supp}(\Theta_{i,k-t})}$. We can then extend $f$ to a function that is strictly less than $1$ on $X_{i}\setminus Y_i$.

We define $g=\Theta_{i,k}- \sum_{t=1}^{M-1} \Theta_{i,k+t}$ on $Y_i$, and we may extend to the entirety of $X_{i}$. Let $g^\prime= \mathrm{max}(g,0)$. $g^\prime$ is therefore $0$ on an open set $U$ containing $\bigcup_{t=1}^{M-1}  \mathrm{supp}(\Theta_{i,k+t})$. Hence, each of these sets is separated from $\overline{ \mathrm{supp}(\Theta_{i,k})}$. Moreover, $g^\prime$ agrees with $\Theta_{i,k}$ on $Y_i$.

Finally, we let $\Theta_{i,k}= \mathrm{min}(f,g^\prime)$. $\Theta_{i,k}$ is therefore $0$ on $\bigcup_{t=1}^{M-1} \overline{ \mathrm{supp}(\Theta_{n-t})}$ and $\overline{ \mathrm{supp}(\Theta_{i,k})}$ is disjoint from $\bigcup_{t=1}^{M-1}  \mathrm{supp}(\Theta_{i,n+t})$. We may check that $\Theta_i(x)$ then satisfies conditions (1)-(3) and (6) for every $x\in X_i$.\end{proof}

\begin{lemma}Suppose that $A$ is a DSH algebra, $M, N \in \mathbb{N}$ with $MN<n_1$, and $f\in A$ is a function such that for every $i,k$, there is an open set $U_{i,k}$ containing $B_{i,k}$ such that when $x\in U_{i,k}$, $f_i(x)$ has zero crosses in positions $k,k+ M, k+2M,\ldots, k+(N-1)M$ and a block point at position $k$. Then there exists a unitary $V\in A$ such that $r(Vf_i V^*(x))\leq r(f_i(x))+2$ for all $i$ and $x\in X_i$ and open sets $U_{i,k}^\prime \supset B_{i,k}$ such that when $x\in U_{i,k}^\prime$, $Vf_i V^*(x)$ has zero crosses in positions $k, k+1, k+2,\ldots, k+(N-1)$.\label{second}\end{lemma}

\begin{proof} We utilize Lemma~\ref{indicator} with $K=\{0\}$ to obtain $\Theta\in A$ such that for each $i,k$, $\Theta_{i,k}=0$ on the complement of the open set $U_{i,k}$. There exist open sets $U_{i,k}^\prime$ such that $U_{i,k}\supseteq\overline{U_{i,k}^{\prime}} \supseteq B_{i,k}$ on which $\Theta_{i,k}=1$. For any consecutive $NM$ entries on the diagonal of any $\Theta_i(x)$, at most one is non-zero and the final $NM-1$ entries are all $0$.

For an index $k\leq n_i-NM$, let $u_k\in C(X_i, M_{n_i})$ be the operator $1_{k-1} \oplus V(\Theta_{i,k}(x)) \oplus 1_{n_i -(NM+k-1)}$ where $V: [0,1]\to M_{NM}$ is the operator given by Lemma~\ref{condense} for $z_1=1, z_2=M+1, \ldots z_N=(N-1)M+1$. Let $u_k$ be the identity for $k>n_i-NM$.

By hypothesis, whenever $\Theta_{i,k}(x)$ is non-zero, $f_i(x)$ has zero crosses at positions $k, k+M, k+2M,\ldots, k+(N-1)M$ and a block point at position $k$. It follows then, from Lemma~\ref{condense}, that $(u_k f_i u_k^*)(x)$ has zero crosses at positions $k, k+1, \ldots ,k+(N-1)$ whenever $x\in U_{i,k}^\prime$. 

Fix $x\in X_i$ and let $K(x)=\{k: \Theta_{i,k}(x)>0\}$ and write this set as $\{k_1, \ldots, k_s\}$. Note that $k_1=1$ and $k_{t+1}-k_t \geq NM$. Consider the operator $V_i^\prime\in C(X_i, M_{n_i})$ given by $V_i^\prime(x) = \prod_{k=1}^{n_i} u_k(x)$. Since $u_k(x)$ is the identity whenever $k\notin K(x)$, we can rewrite this as $V_i^\prime(x)=\prod_{t=1}^s u_{k_t}$. Since the distance between successive $k_t$ is greater than the dimension of $V$, we can give a direct sum decomposition of $V_i^\prime$. 
$$V_i^\prime(x) = \mathrm{diag}(V(\Theta_{i,k_1}(x)), 1_{d_1} , V(\Theta_{i,k_2}(x)) , 1_{d_2},\ldots, V(\Theta_{i,k_s}(x)) , 1_{d_s})$$
where $d_t= k_{t+1}-(k_t+NM)$ for $t<s$ and $d_s=n_i+1- (k_s+NM)$.

Also, because $f_i(x)$ has a block point at position $k_t$ for each $t$, $f_i(x)$ can be written as $Z_1\oplus Z_2\oplus\ldots$ where $Z_t$ is a $k_{t+1}-k_t$ dimensional matrix. Then 
$$u_{k_t} f_i u_{k_t}^*(x)= \mathrm{diag}(Z_1,\ldots, Z_{{t-1}} , (V(\Theta_{k_t}(x)) \oplus 1_{d_t}) Z_{t} (V(\Theta_{k_t}(x))^*\oplus 1_{d_t}) ,\ldots).$$ 
It then follows from Lemma~\ref{condense} applied for each $t$ that $(V_i^{\prime} f_i V_i^{\prime*})(x)$ has zero crosses at positions $k, k+1, \ldots ,k+(N-1)$ whenever $x\in U_{i,k}^\prime$. We also find that $(V(\Theta_{k_t}) \oplus 1_{d_t}) Z_t (V(\Theta_{k_t}) \oplus 1_{d_t})^*$ has diagonal radius at most $r(Z_t)+2$. Therefore $r((V_i^{\prime} f_i V^{\prime*})(x))\leq r(f_i(x))+2$.

It remains to show that $V^\prime=\oplus_i V_i^\prime \in A$. It suffices to show that whenever $y\in Y_i \subset X_i$ with corresponding points $x_{y,1}, \ldots x_{y,s}$, where $x_{y,t}\in X_{i_t}$, $V_i^\prime(y)=\mathrm{diag}( V_{i_1}^\prime(x_{y,1}),\ldots, V_{i_s}^\prime(x_{y,s}))$. Let $B(x)=\{ k: x\in B_{i,k} \}$ and note that $B(x)\subseteq K(x)$. Since $\Theta\in A$, $\Theta_i(y) = \mathrm{diag}( \Theta_{i_1}(x_{y,1}),\ldots, \Theta_{i_s}(x_{y,s}))$. Therefore if $k_{t_1}$ and $k_{t_2}$ are successive elements of $B(x)$, (they need not be successive in $K(x)$), then 
$$\mathrm{diag}(V(\Theta_{i,k_{t_1}}(y)), 1_{d_{t_1}} , \ldots, V(\Theta_{i,k_{t_2-1 }}(y)) , 1_{d_{t_2 -1}})$$
is exactly $V_{i_{t_1}}^\prime (x_{y,t_1})$ so $V_i^\prime(x)$ decomposes as necessary.

\end{proof}

\begin{lemma} Suppose that $A$ is a DSH algebra, $N\in \mathbb{N}$ with $N<n_1$, and $f\in A$ is a function such that for every $i,k$, there is an open set $U_{i,k}$ containing $B_{i,k}$ such that when $x\in U_{i,k}$, $f_i(x)$ has zero crosses in positions $k, k+1,\ldots, k+N-1,$ and $r(f_i(x))<N$. Then there exists a unitary $V\in A$ such that $(fV)_i(x)$ is strictly lower triangular for every $x$.\label{third}\end{lemma}
\begin{proof} We utilize Lemma~\ref{indicator} with $K=\{0\}$ to obtain $\Theta\in A$ such that for each $i,k$, $\Theta_{i,k}=0$ on $U_{i,k}^C$, $\Theta_{i,k}=1$ on $B_{i,k}$ and for any consecutive $N$ entries on the diagonal, at most one is non-zero. For each $i$ we then apply Lemma~\ref{V} to obtain $V_{i}\in C([0,1]^{n_i},M_{n_i})$,  and let $V^\prime= \bigoplus_i V_i\circ\Theta_i$, where we consider $\Theta_{i}(x)$ to be in $[0,1]^{n_i}$. $V_i^\prime$ is then a unitary in $C(X_i, M_{n_i})$.

Suppose that $y\in Y_i\subset X_i$ with corresponding points $x_{y,1}, \ldots x_{y,s}$, where $x_{y,t}\in X_{i_t}$. Consider the set $B=\{k_1=1,\ldots, k_s\}$ of $k$ such that $y\in B_{i,k}$. By our choice of $\Theta$, $\Theta_{i,k}(y)=1$ for every $k\in B$, and by the DSH decomposition, $k_{t+1} -k_t = n_{i_t}$. Since $\Theta\in A$, $\Theta_{i}(y)= \mathrm{diag}(\Theta_1(x_{y,1}),\ldots,\Theta_s(x_{y,s}))$. Therefore by \eqref{Vn},
$$V_i(\Theta_i(x))= \mathrm{diag}( V_{i_{y,s}} (\Theta_{i_1}(x_{y,1})),\ldots, V_{i_s}(\Theta_{i_s}(x_{y,s}))),$$
and similarly
$$V_i^\prime(y)= \mathrm{diag}( V_{i_1}^\prime (x_{y,1}),\ldots, V_{i_s}^\prime(x_{y,s})),$$

Therefore, $V^\prime\in A$. Finally, we apply Lemma~\ref{triangulate} at each point $x$ to conclude that $(fV)_i(x)$ is a lower triangular matrix for every $i$ and $x$.\end{proof}

\section{Spectral Considerations}

The unitary equivalence classes of irreducible representations of $A$ under the kernel-hull topology will be denoted by $\widehat{A}$. The topology on this set is defined by
 
$$\overline{X}=\{ \rho : \ker(\rho) \supseteq \bigcap_{\pi\in X} \ker(\pi)\}.$$

For any cardinal number $n$, $\widehat{A}(n)$ will denote the subspace of $n$-dimensional representation classes.

Let $A$ and $B$ be unital $C^*$-algebras, and suppose that $\phi: A \rightarrow B$ is a homomorphism. Then every irreducible representation of $B$ decomposes into a direct sum of a unique set of irreducible representations of $A$. We define -- via decomposition into irreducible representations -- a set-valued map $\widehat{\phi}: \widehat{B}\rightarrow \mathcal{P}(\widehat{A})$, which ignores multiplicity.

We can characterize simplicity of an inductive limit algebra $A= \lim (A_n , \phi_{n})$ in terms of dual maps. This is almost certainly known; however we have been unable to find a reference. It generalizes Proposition~2.1 of \cite{dadarlat}, and the proof is very similar.
For $n^\prime>n$, we will denote by $\phi_{n^\prime,n}$ the map $\phi_{n^\prime-1}\circ \cdots \circ \phi_n$. 

\begin{theorem} Let $A= \lim (A_n , \phi_{n})$ be a unital $C^*$-algebra and suppose each $\phi_{n}$ is injective. Then the following statements are equivalent
\begin{enumerate} 
\item $A$ is simple.
\item $\forall i\in\mathbb{N}$, open $U\subset \widehat{A}_i$, $\exists j>i$ such that $\widehat{\phi}_{j,i}(\pi)\cap U \neq \varnothing$ for all $\pi\in \widehat{A}_j$.
\item For any non-zero $a\in A_i$, there exists $j>i$ such that $\pi(\phi_{j,i}(a))\neq 0$ for all $\pi\in \widehat{A}_j$.
\end{enumerate}
\label{simple}\end{theorem}
\begin{proof} Fix $i$ and let $U\subset \widehat{A}_i$ be open. For each $j\geq i$, let $F_j= \{ \pi\in \widehat{A}_j: \widehat{\phi}_{j,i}(\pi) \cap U = \varnothing\}$. Consider $\rho\in\widehat{A}_j$ and suppose $\widehat{\phi}_{j,i}(\rho)$ contains a point $\pi\in U$.  Then we may find $f\in A_i$ such that $\pi(f)\neq 0$ and $\sigma(f)=0$ for all $\sigma\in \widehat{A}_i \setminus U$. Therefore $\ker(\rho)$ does not contain  $\bigcap_{\sigma\in F_j} \ker(\sigma)$ since $\pi_{j,i}(f)$ lies in the latter but not the former. It follows that $F_j$ is closed.
Suppose $F_j$ is non-empty for each $j\geq i$. Then for each $j$, we may construct a closed non-zero ideal $I_j= \{a\in A_j: \pi(a)=0\textrm{ } \forall \pi \in F_j\}$. For any $k\geq j$, $\widehat{\phi}_{k,j} \circ \widehat{\phi}_{j,i}=  \widehat{\phi}_{k,i}$. It follows that if $\pi(a)=0$ for all $\pi \in F_j$ then $\pi(\phi_{k,j}(a))=0$ for every $\pi\in F_k$, as $\widehat{\phi}_{k,j}(\pi)\subseteq F_j$. Therefore, $\phi_{k,j}(I_j)\subseteq I_k$. We let $I= \lim (I_n , \phi_n)$, and since $1_A$ cannot lie in $I$, $I$ is a proper ideal of $A$. Therefore (1) implies (2). 
For any non-zero $a\in A_i$, let $U\subset \widehat{A}_i$ be the set $\{\pi : \pi(a)\neq 0\}$. Applying condition (2) immediately implies condition (3). 
When (3) holds, it is clear that for any non-zero $a\in A_i$, there exists $j>i$ such that $\phi_{j,i}(a)$ does not lie in any proper ideal of $A_j$. Therefore any ideal of $A$ which contains a non-zero element of $A_i$ must contain all of $A_j$ for some $j$. Since $A_j$ contains $1_A$, the ideal contains all of $A$.\end{proof}

We restate Lemma 2.1 of \cite{phillips2} in our notation.

\begin{lemma} Suppose that $A$ is a DSH algebra. For $n_i\in\{n_1,\ldots, n_l\}$, $\widehat{A}(n_i)$ is homeomorphic to $X_i \setminus Y_i$. For $n\notin\{n_1,\ldots, n_l\}$, $\widehat{A}(n)$ is empty.\end{lemma}

The following lemma allows us to modify a non-invertible element to produce one with a zero cross.

\begin{lemma} Let $A$ be a DSH algebra, let $\epsilon>0$ and suppose $f\in A$ is non-invertible. Then there exists $f^\prime \in A$ such that $\left|f-f^\prime\right|<\epsilon$ and a unitary $v\in A$ such that for some $i$, $(v f^\prime)_i$ has a zero cross in position $1$ everywhere in some open set $U\subseteq \widehat{A}$ with $\overline{U}\subseteq (X_i\setminus Y_i)$. Moreover, there exists $\Delta\in A$ such that for every $x$, $\Delta(x)$ is a diagonal matrix with entries in $[0,1]$ where $\Delta(x)_{k,k}>0$ implies $f$ has a zero cross at position $k$ and $\Delta(x)_{1,1}=1$ for all $x\in U$. \label{zero}\end{lemma}

\begin{proof}Since $f$ is non-invertible, there exists $X_i$ and a point $x\in X_i$ such that $f(x)$ is a non-invertible matrix. Reducing to a non-invertible summand if necessary, we may assume $x\in X_i\setminus Y_i$. Let $i^\prime$ be the greatest integer such that $x$ lies in the closure of $\widehat{A}(n_{i^\prime})$. Choose an open neighbourhood $U_1\subseteq \widehat{A}$ of $x$ such that $U_1 \cap \widehat{A}(n_{i^{\prime\prime}})=\emptyset$ for all $i^{\prime\prime} > i^{\prime}$. Since $x$ lies in the closure of $(X_{i^\prime}\setminus Y_{i^\prime})\subset \widehat{A}$, there exists a point $y\in Y_{i^{\prime}}$ such that $f_{i^\prime}(y)= M_1\oplus f_{i}(x) \oplus M_2$ for some matrices $M_1,M_2$. We may find a nearby point $x^{\prime}\in U_1 \cap (X_{i^{\prime}}\setminus Y_{i^\prime})$ such that $\left|f_{i^\prime}(x^{\prime})-M_1\oplus f(x) \oplus M_2\right|<\frac{\epsilon}{2}$. 

It follows from Proposition 3.6.3 of \cite{dixmier} that $\widehat{A}_{n_{i^{\prime}}}\bigcap U_1$ is open in $U_1$. 

Since $U_1$ does not lie in the closure of $X_{i^{\prime\prime}}\setminus Y_{i^{\prime\prime}}$ for any $i^{\prime\prime} > i^{\prime}$, we may, by assuming that each $Y_{i^{\prime\prime}}$ has empty interior, conclude that there is no point in any $Y_{i^{\prime\prime}}$ has a corresponding point $x_{y,t}\in U_1$. We may then perturb $f$ on $U_1\cap (X_{i^{\prime}}\setminus Y_{i^{\prime}})$ without altering $f_{i^{\prime\prime}}$ for any $i^{\prime\prime} > i^{\prime}$. In particular, we can obtain a function $f^\prime$ such that $ \left|f-f^\prime\right|<\epsilon$ which is equal to $M_1\oplus f_i(x) \oplus M_2$ on a neighbourhood $U_2$ of $x^{\prime}$  with $\overline{U_2}\subseteq (X_{i^\prime}\setminus Y_{i^\prime}) \cap U_1$. Then there exists a unitary matrix $V$ such that $V(M_1\oplus f(x) \oplus M_2)$ has a zero cross in position $1$.  By considering a path of unitaries in $U(n_{i^\prime})$ between $V$ and the identity, we may  extend $V$ to a unitary $v$ in $A$ which is constantly equal to $V$ on $U_2$ and is the identity everywhere outside $(X_{i^\prime}\setminus Y_{i^\prime})\cap U_1$. 

To construct $\Delta$, begin by setting $\Delta_{i^{\prime\prime}}=0$ for $ i^{\prime\prime}< i^{\prime}$. On $X_{i^{\prime}} \setminus Y_{i^{\prime}}$, we let $\Delta_{i^\prime}(x)$ be zero everywhere except on $U_2$ and always zero in every coordinate  of $\Delta_{i^\prime}$ except $\Delta_{i^\prime}(x)_{1,1}$. We consider $\Delta_{i^\prime}(x)_{1,1}$ to be in $C(X_{i^\prime}, [0,1])$ and choose a function such that $\Delta_{i^\prime}(x)_{1,1}=1$ everywhere on some non-empty open set $U_3\subseteq U_2$ and $0$ everywhere outside $U_2$. Since $\overline{U_2}$ is separated from $X_{i^{\prime\prime}}$ for any $i^{\prime\prime} > i^{\prime}$, we may set $\Delta_{i^{\prime\prime}}$ to be $0$ on $X_{i^{\prime\prime}}$.

We find that $f^\prime$,$v$, $\Delta$, and $U_3$ then have the desired properties.
\end{proof}

Note that the set $U_3$ is, using the bijection we have established between $(X_i\setminus Y_i)$ and $\widehat{A}(n_i)$, also an open set in $\widehat{A}$.

Suppose we have a limit of DSH algebras $A_j$. In the following, we must eliminate the case in which the representations of $A_j$ do not grow in dimension. By \ref{simple}, when the sequence is injective, such a limit can only be simple if the algebras $A_j$ are all of finite spectrum. Therefore it is enough to assume each $A_j$ is infinite dimensional.

\begin{lemma} Suppose that $A= \mathrm{lim}(A_j,\phi_{j+1,j})$ where $A_j$ is infinite dimensional and DSH, $\phi_{j+1,j}$ is diagonal and $A$ is simple. Suppose that $f$ is a non-invertible element in $A_j$ and let $\epsilon>0$. Then there exists $f^\prime\in A_j$ with $\left|f-f^\prime\right|<\epsilon$ and a constant $M$ such that for any $N\in\mathbb{N}$, there exist $j^\prime >j$, unitaries $V_1,V_2 \in A_{j^\prime}$ such that for every $X_i^{j^\prime}, k\leq n_i^{j^\prime}$, there is an open set $U_{i,k}$ containing $B_{i,k}^{j\prime}\subset X_i^{j^\prime}$, such that for each $x\in U_{i,k}$,  $(V_1 \phi_{j^\prime,j}(f^\prime)V_2)_i(x)$ has zero crosses in positions $k,M+k,2M+k,\ldots,(N-1)M+k$ and $r((V_1 \phi_{j^\prime,j}(f^\prime)V_2)_i(x))<R+M-1$, where $R$ is the size of the largest representation of $A_j$.\label{first}\end{lemma}

\begin{proof} Since $f$ is non-invertible, we may find $f^\prime,v,\Delta\in A_j$ and $U\subseteq\widehat{A}_j$, as in Lemma~\ref{zero}. By Theorem~\ref{simple}, we may choose $j^{\prime\prime}$ such that for every point $x\in \widehat{A}_{j^{\prime\prime}}$, $\widehat{\phi}_{j^{\prime\prime},j}(x)$ contains a point in $U$. If $n$ is the the largest dimension of the irreducible representations of $A_{j^{\prime\prime}}$ then we choose $M=2n$.

Consider $\Delta^{\prime\prime}$, the image of $\Delta$ in $A_{j^{\prime\prime}}$. Because $\phi_{i^{\prime\prime},i}$ is diagonal, $\Delta_i^{\prime\prime}(x)$ is a diagonal matrix for every point $x$ in every $X_i^{j^{\prime\prime}}$, and its entries are in $[0,1]$. It follows from our choice of $j^{\prime\prime}$ that for every $x\in X_i^{j^{\prime\prime}}$, $\Delta_i^{\prime\prime}(x)$ has an entry that is equal to $1$ somewhere on the diagonal coming from the point in $\widehat{\phi}_{j^{\prime\prime},j}(x)$ that lies in $U$.

Let $N\in \mathbb{N}$ be arbitrary. We can choose $j^\prime$ such that $n_1$, the smallest dimension of a representation of $A_{j^\prime}$, is arbitrarily large. In particular, we choose $n_1$ to be at least $NM$.

Let $\Delta^\prime=\phi_{j^\prime,j}(\Delta)$, and consider $\Delta^\prime$ as the diagonal image of $\Delta^{\prime\prime}$. As before, we condense our notation, writing $\Delta_{i,k}^\prime(x)$ rather than $\Delta_i^\prime(x)_{k,k}$. For any $i$, $x\in X_i^{j^\prime}$ and any entry $k< n_i^{j^\prime}-M$, there must be an entry $k^\prime \in \{k,k+1,\ldots, k+M-1\}$ for which $\Delta_{i,k}^\prime(x)=1$.

Let $g= \phi_{j^\prime,j} (v f^\prime)$. Since $\phi_{j^\prime,j}$ is diagonal, $g_i(x)$ is block diagonal for every $x$ in every $X_i^{j^\prime}$, and the maximum block size is $R$, so $r(g_i(x))\leq R$. Moreover, whenever $\Delta_{i,k}^\prime(x)>0$, $g(x)$ has a zero cross at position $k$.

We will need a second set of indicator functions. We construct $\Theta\in A_{j^\prime}$ as in Lemma~\ref{indicator} for our already given value of $M$ and $K=\{0,M, \ldots, (N-1)M\}$. For each $ k,X_i^{j^\prime}$, there is an open set $U_{i,k}$ containing $B_{i,k}^{j^\prime}$, such that for each $x\in U_{i,k}$, $\Theta_{i,k}(x) = \Theta_{i,k+M}(x)= \cdots = \Theta_{i,k+(N-1)M}(x)=1$. For $k\in 0,\ldots, M-1$, $\Theta_{i,n_i-k}$ is the constant zero function. Finally, for arbitrary $x,k$, at most one of $\Theta_{i,k}(x),\ldots,\Theta_{i,k+M-1}(z)$ is non-zero.

We can now construct a unitary element $V$ such that $VgV^*$ has zero crosses where required and a bounded diagonal radius. On $X_i^{j^\prime}$, we let $u_k(x)\in C(X_i^{j^\prime},M_{n_i})$ be given by 
$$u_k(z)=u_{(k\;k+1)}(\Theta_{i,k}(x)\Delta_{i,k+1}^\prime(x)) \cdots u_{(k\;k+M-1)}(\Theta_{i,k}(x)\Delta_{i,k+M-1}^\prime).$$

We then define $V_i\in C(X_i^{j^\prime},M_{n_i^{j^\prime}})$ by $V_i= \prod_{k=1}^{n_i-(M-1)} u_k$. When $\Theta_k(x)=0$, $u_i(x)$ is the identity, so $V_i$ can be rewritten at each $x$ as a product of $u_{k_1}, \ldots, u_{k_n}$ where, by the definition of $\Theta$, $k_{i+1}-k_i \geq M$. When $x\in U_{i,k}$, $\Theta_{i,k}(x) = \Theta_{i,k+M}(x)= \cdots = \Theta_{i,k+(N-1)M}(x)=1$ and there is always at least one entry in $\Delta_{i,k+aM}^\prime(x),\ldots,\Delta_{i,k+(a+1)M-1}^\prime(x)$ that is equal to $1$ for $a\in \{0,\ldots N-1\}$, hence by Lemma~\ref{block1}, $(VgV^*)_i(x)$ has zero crosses at positions $k, k+M, \ldots, k+(N-1)M$. Also by Lemma~\ref{block2}, we have that $r((VgV^*)_i(x))\leq R+M-1$ for all $x$.

It remains to show that $V= \bigoplus_i V_i\in A_{j^\prime}$. For $x\in Y_i^{j^\prime}\subseteq X_i^{j^\prime}$, with corresponding points $x_{1}, \ldots x_{s}$, where $x_{m} \in X_{i_m}^{j^\prime}$ and let $K=\{k_1=1,\ldots, k_s\}$ be the set of $k$ such that 
$x\in B_{i,k}^{j^\prime}$. We decompose the product $V_i= \prod_{k=1}^{n_i-(M-1)} u_k$ into $\prod_{m=1}^s \prod_{k=k_m}^{k_{m+1}-(M-1)} u_{k}$ where we let $k_{s+1}=n_i+1$. Consider then the product $v_m(x)=\prod_{k=k_m}^{k_{m+1}-(M-1)} u_{k}$,
$$v_m(x)=\prod_{k=k_m}^{k_{m+1}-(M)} \prod_{a=1}^{M-1} u_{( k\;k+a)}(\Theta_{i,k} (x)\Delta_{i,k+a}^\prime(x)),$$
and note first that because $\Theta_{i,k}(x)=0$ for $k\in k_m-M+1, \ldots, k_m-1$, this is the same as taking the product in $k$ to $k_{m+1}-1$. $v_m(x)$ is block diagonal of the form $1_{k_m} \oplus B \oplus 1_{n_i-k_{m+1}}$ where $B$ is a square $k_{m+1}-k_m$ dimensional matrix. Furthermore, the matrix $B$ is given by 
$$u_{(k\;k+1)}(\Theta_{i,k_m+k}(x)\Delta_{i,k_m+k+1}^\prime(x)) \cdots u_{(k\;k+M-1)}(\Theta_{i,k}(x)\Delta_{i,k+M-1}^\prime(x)).$$
Since we have constructed $\Theta$ and $\Delta^\prime$ as elements of $A_{j^\prime}$, $\Theta_{i,k_m+k}(x)=\Theta_{i_m,k}(x_m)$ and $\Delta_{i,k_m+k}^\prime(x)=\Delta_{i_m,k}^\prime(x_m)$ where $k< k_{m+1}-k_m$. We have that $B= V_{i_m}(x_m)$ and $V_i= V_{i_1}(x_1)\oplus\cdots\oplus V_{i_{s}}(x_s)$. Therefore $V\in A_{j^\prime}$.

Finally, consider the unitaries $V \phi_{j^\prime,j}(v)$ and $V^*$. These unitaries satisfy the lemma, since $V \phi_{j^\prime,j}(vf^\prime) V^*= VgV^*$.\end{proof}

\section{The Main Theorem}

\begin{theorem} Suppose that $A=\mathrm{lim}(A_j, \phi_{j})$ is a simple limit of DSH algebras with diagonal maps. Then $A$ has stable rank one.\label{main}\end{theorem}

\begin{proof} If $A$ is a limit of finite dimensional algebras then we are done, so we may assume all $A_j$ are infinite dimensional. Moreover, by Lemma~\ref{quotient}, we may assume the maps are injective.

Let $\epsilon>0$ and $a\in A$. We will show that there is an invertible element $a^\prime \in A$ such that $\left| a - a^\prime\right|<\epsilon$. We find $A_j$ and an element $f\in A_j$ such that $\left|\iota(f)-a\right|<{\epsilon}/{4}$, where $\iota$ is the inclusion of $A_j$ into $A$. If $f$ is invertible we are done, so we assume it is not. Denote by $R$ the largest dimension of a representation of $A_j$. By \ref{first}, we have $f^\prime\in A_j$ such that $\left|f-f^\prime\right|<{\epsilon}/{4}$ and $M\in\mathbb{N}$ such that for any $N\in\mathbb{N}$, we have $j^\prime > j$ and unitaries $V_1,V_2 \in A_{j^\prime}$ such that in the construction of $A_{j^\prime}\subset \bigoplus_i C(X_i, M_{n_i})$, for every $i, k$, there is an open set $U_{i,k}$ containing $B_{i,k}^{j^\prime}$, such that for each $x\in U_{i,k}$,  $G_i(x)$ has zero crosses in positions $k,M+k,2M+k,\ldots,(N-1)M+k$ and $r(G_i(x))<R+M$, where $G=V_1 \phi_{j^\prime,j}(f^\prime)V_2\in A_{j^\prime}$. In particular, we choose $N>R+M+2$.

We apply Lemma~\ref{blockpoint} to obtain $G^\prime \in A_{j^\prime}$ such that $\left|G-G^\prime\right|<{\epsilon}/{4}$, and there exist open sets $U_{i,k}^\prime$ containing each $B_{i,k}^{j^\prime}$ such that for each $x\in U_{i,k}$,  $G_i^\prime(x)$ has zero crosses in positions $k,M+k,2M+k,\ldots,(N-1)M+k$ and a block point at position $k$. Moreover, for every $x\in X_i$, $r(G^\prime_i(x))<R+M$.  We apply Lemma~\ref{second} to obtain $V_3$ such that there exist open sets $U_{i,k}^{\prime\prime}$ on which $(V_3 G^{\prime} V_3^*)_i$ has zero crosses at positions $1,\ldots, N$ where $N>R+M+2$ and for every $x\in X_i$, $r((V_3 G^\prime V_3^*)_i(x))<R+M+2$. Then by Lemma~\ref{third} there exists a unitary $V_4\in A_{j^\prime}$ such that $(V_3 G^\prime V_3^*V_4)_i(x) $ is lower triangular at every $x\in X_i$. Therefore $ V_3 G^\prime V_3^*V_4$ is nilpotent and by \cite{rordam} there exists an invertible element $G^{\prime\prime}\in A_{j^\prime}$ such that $\left| V_3 G^{\prime} V_3^*V_4-G^{\prime\prime}\right|<{\epsilon}/{4}$. Then $\left|\phi_{j^\prime,j}(f) -  V_1^* V_3^* G^{\prime\prime} V_4^*V_3 V_2^*\right|<{3\epsilon}/{4}$, and so $\left|a -  \iota(V_1^* V_3^* G^{\prime\prime}V_4^* V_3 V_2^*)\right|<\epsilon$ and the latter is invertible.\end{proof}

\begin{cor} For any compact Hausdorff space and any minimal homeomorphism $\sigma: X\to X$, the crossed product $C^*(\mathbb{Z}, X, \sigma)$ has stable rank one.\end{cor}
\begin{proof} Lemma~\ref{decomp} shows that $A_y$ can be constructed as a simple injective limit of DSH algebras with diagonal maps. It follows from Theorem~\ref{main} that $A_y$ has stable rank one. It is shown in \cite{archey} that $A_y$ is centrally large in $C^*(\mathbb{Z}, X, \sigma)$. By Theorem 6.3 of \cite{archey}, any infinite dimensional simple separable unital $C^*$-algebra with a centrally large subalgebra of stable rank one has stable rank one.\end{proof}

\end{document}